\documentclass{amsart}
\usepackage{amssymb}
\usepackage{graphicx}
\usepackage[all]{xy}
\usepackage{hyperref}
\usepackage{a4wide}
\begin{document}

\fontsize{12}{18}\selectfont

\renewcommand{\labelenumi}{(\alph{enumi})}

\newtheorem{thmint}{Theorem}
\newtheorem{lemma}{Lemma}[section]
\newtheorem{theorem}[lemma]{Theorem}
\newtheorem*{thm}{Theorem}
\newtheorem{corollary}[lemma]{Corollary}
\newtheorem{fact}[lemma]{Fact}
\newtheorem{proposition}[lemma]{Proposition}
\theoremstyle{remark}
\newtheorem{claim}[lemma]{Claim}
\newtheorem{remark}[lemma]{Remark}
\newtheorem{question}[lemma]{Question}
\newtheorem{conjecture}[lemma]{Conjecture}
\newtheorem*{rem}{remark}
\newtheorem{example}[lemma]{remark}
\newtheorem{idea}[lemma]{Idea}
\theoremstyle{definition}
\newtheorem{condition}[lemma]{Condition}
\newtheorem{definition}[lemma]{Definition}

\newenvironment{label1}{
\begin{enumerate}
\renewcommand{\theenumi}{\arabic{section}.\arabic{lemma}\alph{enumi}}
\renewcommand{\labelenumi}{(\arabic{section}.\arabic{lemma}\alph{enumi})}
}{\end{enumerate}\renewcommand{\theenumi}{\alph{enumi}}}{}
\let\origwr=\wre
\def\wre{\mathop{\rm wr}\nolimits}
\renewcommand{\phi}{\varphi}
\newcommand{\Ind}{{\rm Ind}}
\def\res{{\rm res}}
\def\irr{{\rm irr}}
\def\normal{\triangleleft}
\newcommand{\gal}{\textnormal{Gal}}
\newcommand{\bfsigma}{\mbox{\boldmath$\sigma$}}
\def\ab{{\rm ab}}
\def\Fhat{{\hat F}}
\def\Fgag{{\bar{F}}}
\def\Ehat{{\hat E}}
\def\Egag{{\bar{E}}}
\def\Egal{{\tilde{E}}}
\def\calE{\mathcal{E}}
\def\calS{\mathcal{S}}
\def\Lgal{{\tilde{L}}}
\newcommand{\bbA}{\mathbb{A}}
\newcommand{\bbF}{\mathbb{F}}
\newcommand{\bbQ}{\mathbb{Q}}
\newcommand{\bbZ}{\mathbb{Z}}
\newcommand{\bbZgal}{{\tilde \bbZ}}
\newcommand{\bbQgal}{{\tilde \bbQ}}
\newcommand{\Kgal}{{\tilde K}}
\def\mugag{{\bar \mu}}
\def\alphahat{{\hat \alpha}}
\def\phihat{{\hat \varphi}}
\def\Phihat{{\hat \Phi}}
\def\thetahat{{\hat \theta}}
\newcommand{\bfa}{\mbox{{\rm \textbf{a}}}}
\newcommand{\bft}{\mbox{{\rm \textbf{t}}}}
\newcommand{\bfu}{\mbox{{\rm \textbf{u}}}}
\newcommand{\bfT}{\mbox{{\rm \textbf{T}}}}
\newcommand{\bfX}{\mbox{{\rm \textbf{X}}}}
\CompileMatrices 
\title[On PAC extensions]{On pseudo algebraically closed extensions of fields}%
\author{Lior Bary-Soroker}%
\address{The Raymond and Beverly Sackler school of mathematical sciences, Tel-Aviv University}%
\email{barylior@post.tau.ac.il}%


\date{\today}%
\thanks{This work is a part of the author's PhD thesis done at Tel Aviv University under the supervision of Prof.\ Dan Haran }
\thanks{
This research was partially supported by the ISRAEL SCIENCE FOUNDATION (Grant No.
343/07).}%
\begin{abstract}
The notion of `Pseudo Algebraically Closed (PAC) extensions' is a
generalization of the classical notion of PAC fields.
In this work we develop a basic machinery to study PAC extensions.
This machinery is based on a generalization of embedding problems to
field extensions. The main goal is to prove that the Galois closure
of any proper separable algebraic PAC extension is its separable
closure.
As a result we get a classification of all finite PAC extensions
which in turn proves the `bottom conjecture' for finitely generated
infinite fields.

The secondary goal of this work is to unify proofs of known results
about PAC extensions and to establish new basic properties of PAC
extensions, e.g.\ transitiveness of PAC extensions.
\end{abstract}
\maketitle

\section{Introduction and Results}
This work concerns Pseudo Algebraically Closed (PAC) extensions of
fields, and especially the Galois structure of PAC extensions. We
start by a short survey on this notion and its importance. Then we
discuss the results and methods of this work.

\subsection{Pseudo Algebraically Closed Extensions}
A field $K$ is called \textbf{PAC} if it has the following geometric
feature: $V(K)\neq \emptyset$ for any non-void absolutely
irreducible variety $V$ which is defined over $K$. In
\cite{JardenRazon1994} Jarden and Razon generalize this classical
notion to a field $K$ and a subset $K_0$ (in this work $K_0$ will
always be a field, unless otherwise stated): We call $K/K_0$
\textbf{PAC extension} (or just say that $K/K_0$ is PAC) if for
every absolutely irreducible variety $V$ of dimension $r\geq 1$
which is defined over $K$ and every separable dominating rational
map $\nu \colon V \to \bbA^r$ defined over $K$ there exists $a\in
V(K)$ for which $\nu(a) \in K_0^r$. (See
Proposition~\ref{prop:DefinitionPACextension} for equivalent
definitions of PAC extensions in terms of polynomials and of
places.) For example, every separably closed field is PAC over any infinite subfield, and for every PAC field K, the trivial extension K/K is PAC.

The original motivation of Jarden and Razon for this definition is
related to a generalization of Hilbert's 10th problem to `large'
algebraic rings. The problem asks whether there exists an algorithm
that determines whether a system of polynomial equations over $\bbZ$
has a solution in $\bbZ$. Matijasevich gave a negative answer to
that problem relying on the works of Davis, Putnam, and J.~Robinson
since the 1930s (see \cite{Matijasevich1993}).

A natural generalization of Hilbert's 10th problem is to consider
solutions in other rings. In \cite{Rumely1986} Rumely establishes a
local-global principle for the ring $\bbZgal$ of all algebraic
integers, and deduces from it a positive answer of Hilbert's 10th
problem for that ring. (It is interesting to mention that for the
ring $\bbQ$ the problem is still open.)

In \cite{JardenRazon1998} Jarden and Razon extend Rumely's local
global principle to the ring of integers $R$ of an algebraic field
$K$, for `almost all' $K\subseteq \tilde \bbQ$ whose absolute Galois
group $\gal(K)$ is finitely generated. Their key idea is to use the
fact (proven by them in \cite{JardenRazon1994}) that these fields
are PAC over $\bbZ$, and hence over $R$. Then, following Rumely,
they deduce a positive answer to Hilbert's 10th problem for $R$.

In \cite{JardenRazon_ms} Jarden and Razon continue their work, and
establish Rumely's local-global principle for smaller rings. They
also deal with the positive characteristic case, and strengthen the
local-global principle itself.

The applications of PAC extensions are not restricted to Hilbert's
10th problem and Rumely's local-global principle, and indeed some
other applications recently appeared in the literature. Some
examples are: (1) new constructions of Hilbertian domains
\cite{Razon1997}; (2) an analog of Dirichlet's theorem on primes in
arithmetical progressions for a polynomial ring in one variable over
some infinite fields \cite{Bary-SorokerDirichlet}; (3) the study of
the question: When is a non-degenerate quadratic form isomorphic to
a scaled trace form? \cite{Bary-SorokerKelmer}.

\subsection{The Galois Closure of PAC Extensions}
In \cite{JardenRazon1994} (where PAC extensions first appear) Jarden
and Razon find some Galois extensions $K$ of $\bbQ$ such that $K$ is
PAC as a field but $K$ is a PAC extension of no number field. (For
this a heavy tool is used, namely Faltings' theorem.) Then they
ask whether this is a coincidence or a general phenomenon.

In \cite{Jarden2006} Jarden settles this question by showing that
the only Galois PAC extension of an arbitrary number field is its
algebraic closure. Jarden does not use Faltings' theorem, but
different results. Namely, Razon's splitting theorem (see
Theorem~\ref{thm:intSplitting}), Frobenius' density theorem,
Neukirch's characterization of $p$-adically closed fields among all
algebraic extensions of $\bbQ$, and also the special property of
$\bbQ$ that it has no proper subfields(!). For that reason Jarden's
method is restricted to number fields.

The next step is to consider a finitely generated infinite field
$K_0$. Elaborating the original method of Jarden-Razon, i.e.\ using
Faltings' theorem (and the Grauert-Manin theorem in positive
characteristic), Jarden and the author generalize Jarden's result to
$K_0$ \cite{Bary-SorokerJarden}.

In this work we further generalize this theorem to the most general
case, where $K_0$ is an arbitrary field. Namely we prove that the
only Galois PAC extensions are the trivial ones (see
Theorem~\ref{thm:intmineGalois} below). Our proof is based on the
\emph{lifting property} which will be discussed later in the
introduction and the realization of wreath products in fields. Thus
it uses no special features of finitely generated fields, and for
that reason it applies to any $K/K_0$.

\begin{thmint}\label{thm:intmineGalois}
Let $K/K_0$ be a proper separable algebraic PAC extension. Then the
Galois closure of $K/K_0$ is the separable closure of $K_0$. \\
In particular, if $K/K_0$ is a Galois PAC extension, then either
$K=K_0$ or $K=K_s$.
\end{thmint}

It is important to note in this stage, that there are a lot of PAC extensions, which are not Galois. 
Let $e$ be a positive integer. If $K_0$ is a countable Hilbertian
field, then for almost all $\bfsigma = (\sigma_1, \ldots,
\sigma_e)\in \gal(K_0)^e$ (in the sense of the Haar measure)
\[
K_{0s}(\bfsigma) = \{ x\in K_{0s}\mid \sigma_i(x) = x, \ \forall
i\}
\]
is a PAC extension of $K_0$ \cite{JardenRazon1994}, and hence of any
subfield of $K_{0s}(\bfsigma)/K_0$. Using this result, in \cite{Bary-SorokerDirichlet,Bary-SorokerKelmer} many PAC extensions are constructed. For example if $K_0$ is a pro-solvable extension of a countable Hilbertian field, then there exists a PAC extension $K/K_0$ such that the order of $\gal(K)$ (as a supernatural number) is $\prod_{p}p^\infty$. In a sequel \cite{Bary-SorokerProjectivePairs}, the author studies the group theoretic properties of the pair of profinite groups $\gal(K) \to \gal(K_0)$ and using the transitivity of PAC extensions appear here constructs new PAC extensions. For example, for any projective profinite group $P$ of rank at most countable, $\bbQ_{ab}$ has a PAC extension with absolute Galois group $P$. It is open whether a finitely generated field $K_0$ has a PAC extension whose absolute Galois group is not finitely generated \cite{Bary-SorokerJarden}, Conjecture~7. 

Note that Theorem~\ref{thm:intmineGalois} generalizes the following

\begin{thmint}[Chatzidakis {\cite{FriedJarden1986}, Theorem~24.53}]
\label{thm:intChatzidakis} Let $K_0$ be a countable Hilbertian
field. Then for almost all $\bfsigma \in \gal(K_0)^e$ the field
$K_{0s} (\bfsigma)$ is Galois over no proper subextension of
$K_{0s}(\bfsigma)/K_0$.
\end{thmint}

We shall also prove that if $K_0$ is a finitely generated infinite
field and $e\geq 1$ an integer, then for almost all $\bfsigma\in
\gal(K_0)^e$ the field $K_{0s}(\bfsigma)$ is a Galois extension of
no proper subfield $K$ (i.e.\ we remove the restriction $K_0
\subseteq K$).

It is interesting to note that if $M$ is a minimal Henselian field,
then $M$ is a Galois extension of no proper subfield. This result
follows from a theorem of Schimdt, that was generalized by Engler,
and later was reproved by Jarden, see \cite{Engler} .

\subsection{Finite PAC Extensions and the `Bottom Conjecture'}
We classify all finite PAC extensions:
\begin{thmint}\label{thm:finitePACext}
Let $K/K_0$ be a finite extension. Then $K/K_0$ is PAC if and only
if one of the following holds.
\begin{enumerate}
\item
$K_0$ is a PAC field and $K/K_0$ is purely inseparable.
\item
$K_0$ is real closed and $K$ is its algebraic closure.
\end{enumerate}
\end{thmint}

Let $K_0$ be a countable Hilbertian field. Similarly to
Theorem~\ref{thm:intmineGalois}, Theorem~\ref{thm:finitePACext} can
be applied to the field $K_{0s}(\bfsigma)$, for almost all $\bfsigma
\in \gal(K_0)^e$. This reproves the `bottom theorem' in the
countable case:

\begin{thmint}[Haran {\cite{FriedJarden2005}, Theorem~18.7.7}]
\label{thm:intbottom} Let $K_0$ be a Hilbertian field and $e\geq 1$
an integer. Then for almost all $\bfsigma \in \gal(K)^e$, 
$K_{0s}(\bfsigma)$ is a finite extension of no proper subfield that
contains $K_0$.
\end{thmint}

Moreover, as in the case of Galois extensions, we strengthen this
result and prove the `bottom conjecture' \cite{FriedJarden2005},
Problem~18.7.8 for finitely generated fields. For the precise
formulation see Conjecture~\ref{conj:bottom}.

It is interesting to note that the theorems of Schmidt, Engler, and
Jarden discussed before, also imply that a minimal Henselian field
is a finite extension of no proper subfield.

\subsection{Double Embedding Problems}
Before continuing with results, we wish to briefly discuss the
methods.

Arithmetic and geometric properties of a field $K$ relate to
properties of the absolute Galois group $\gal(K)$ of $K$. This group
is equipped with the Krull topology, which makes it a profinite
group. A fundamental tool in the study of profinite groups is the
notion of finite embedding problems.

Here is a nice example of the above relation. A profinite group is
the absolute Galois group of some PAC field if and only if all
finite embedding problems for this group are weakly solvable. See
Ax' theorem \cite{FriedJarden2005}, Theorem 11.6.2, and
Lubotzky-v.d.~Dries' theorem \cite{FriedJarden2005},
Corollary~23.1.2.

First we generalize the notion of an embedding problem for a field
$K$ to a \emph{double embedding problem} for an extension $K/K_0$.
If the former is defined w.r.t.\ $\gal(K)$, then the later is
defined w.r.t.\ the restriction map $\gal(K)\to \gal(K_0)$. Roughly
speaking, a double embedding problem for $K/K_0$ consists on two
embedding problems, the `lower' for $K$ and the `upper' for $K_0$,
which are compatible w.r.t.\ the restriction map.

We characterize PAC extensions in terms of special solutions of
finite double embedding problems -- \emph{geometric solutions}.
Those are weak solutions that are induced by some rational point of
a variety, or in a different terminology, by a place of a finitely
generated regular extension (see
Section~\ref{sec:geometricsolution}).

The key property that PAC extensions satisfy is the \emph{lifting
property} (Proposition~\ref{prop:ExtensionSoltoGeoSol_DEP}). This
lifting property asserts that any weak solution of the lower
embedding problem can be extended to a geometric solution of the
double embedding problem, provided some rationality assumption holds.

\subsection{Transitiveness}
We prove that the PAC property of algebraic extensions is
transitive.

\begin{thmint}\label{thm:transitivity}
Let $K_0\subseteq K_1 \subseteq K_2$ be a tower of separable
algebraic extensions. If both $K_2/K_1$ and $K_1/K_0$ are PAC
extensions, then so is $K_2/K_0$.
\end{thmint}

This fundamental property easily follows from the lifting property.
To the best of our knowledge it did not appear in the literature
before.  Moreover, this result together with an analog of the
Ax-Lubotzky-v.d.~Dries result described above leads to a new
construction of PAC extensions. This will be dealt with in \cite{Bary-SorokerProjectivePairs}.

\subsection{Descent of Galois Groups}
In \cite{Razon2000} Razon proves for a PAC extension $K/K_0$ that
every separable extension $L/K$ descends to a separable extension
$L_0/K_0$:

\begin{thmint}[Razon]\label{thm:intSplitting}
Let $K/K_0$ be a PAC extension and $L/K$ a separable algebraic
extension. Then there exists a separable algebraic extension
$L_0/K_0$ that is linearly disjoint from $K$ over $K_0$ such that
$L=L_0 K$.
\end{thmint}

We generalize Razon's result and get the following stronger descent
result.
\begin{thmint}\label{thm:intmineDescent}
Let $K/K_0$ be a PAC extension and $L/K$ a finite Galois extension.
Assume $\gal(L/K) \leq G_0$, where $G_0$ is regular over $K_0$. Then
there exists a Galois extension $L_0/K_0$ such that
$\gal(L_0/K_0)\leq G_0$ and $L = L_0 K$.
\end{thmint}
(Here $G_0$ is regular over $K_0$ if there exists a Galois extension
$F_0/K_0(t)$ with Galois group isomorphic to $G_0$ and such that
$F_0$ is regular over $K_0$.)

Razon's theorem follows from Theorem~\ref{thm:intmineDescent}
applied to the group $G_0=S_n$ (for full details see the proof in
Section~\ref{sec:descent}). Note that the original approach of Razon
to Theorem~\ref{thm:intSplitting} is similar to our proof but very
specific: One only considers the regular realization of $G_0=S_n$
generated by the generic polynomial $f(T_1,\ldots,T_n,X) = X^n + T_1
X^{n-1} + \cdots + T_n$.

Let us explain the name `descent' attached to
Theorem~\ref{thm:intmineDescent}. If a finite Galois group
$G=\gal(L/K)$ over $K$ is regular over $K_0$, then, by taking $H=G$
in Theorem~\ref{thm:intmineDescent}, we get that $G$ occurs over
$K_0$ (since $G=\gal(L_0/K_0)$ in that case). Thus $G$
\emph{descends} to a Galois group over $K_0$.

As a consequence of this and of the fact that finite abelian groups
are regular over any field, we get, for example, that
\[K^{\ab} =K K_0^{\ab}.\]
Here the superscript `ab' denotes the maximal abelian extension.

\subsection{Acknowledgments}
The author is indebted to D.~Haran and M.~Jarden for many
discussions and suggestions concerning this work. The author thanks
A.~Fehm and E.~Paran for valuable remarks on an earlier version of
this paper, C.~Meiri for his suggestion that simplified the diagram of a double embedding problem and finally the anonymous referee for his helpful comments.

\section{Preliminaries}
In this section we define the notion of geometric solutions of an
embedding problem and set up the notation and the necessary
background needed for this work.
\subsection{Embedding Problems}
Let $K$ be a field. Then $K_s$ denotes its separable closure and
$\Kgal$ its algebraic closure. The absolute Galois group of $K$ is
denoted by $\gal(K)$, i.e.\ $\gal(K) = \gal(K_s/K) = {\rm
Aut}(\Kgal/K)$. Recall that an \textbf{embedding problem} for
$\gal(K)$ (or equivalently for $K$) is a diagram
\begin{equation}\label{eq:EP}
\xymatrix{%
    &\gal(K)\ar[d]^{\mu}\ar@{.>}[dl]_{\exists\theta?}\\
G\ar[r]^\alpha
    &A,
}%
\end{equation}
where $G$ and $A$ are profinite groups and $\mu$ and $\alpha$ are
(continuous) epimorphisms. In short we write $(\mu,\alpha)$ for
\eqref{eq:EP}.

A \textbf{solution} of $(\mu,\alpha)$ is an epimorphism
$\theta\colon \gal(K)\to G$ such that $\alpha\theta = \mu$. If
$\theta$ is a homomorphism that satisfies $\alpha\theta = \mu$ but
is not necessarily surjective, we say that $\theta$ is a
\textbf{weak solution}. In particular, a profinite group $G$ is a
quotient of $\gal(K)$ if and only if the embedding problem
$(\gal(K)\to 1, G\to 1)$ is solvable.

If $G$ is finite (resp.\ $\alpha$ group theoretically splits), we
say that the embedding problem is \textbf{finite} (resp.\
\textbf{split}).

The following lemma gives an obvious, but useful, criterion for a
weak solution to be a solution (i.e.\ surjective).

\begin{lemma}\label{lem:criterionforproperness}
A weak solution $\theta\colon \gal(K) \to G$ of an embedding problem
\eqref{eq:EP} is a solution if and only if $\ker(\alpha) \leq
\theta(\gal(K))$.
\end{lemma}

\begin{proof} Suppose $\ker(\alpha)\leq \theta(\gal(K))$.
Let $g\in G$, put $a=\alpha(g)$, and let $f\in \mu^{-1} (a)$.
Then $\theta(f)^{-1} g \in \ker(\alpha)\leq \theta(\gal(K))$,
and hence $g\in \theta(\gal(K))$. The converse is immediate.
\end{proof}

Two embedding problems $(\mu\colon \gal(K)\to A, \alpha\colon
G\to A)$ and $(\nu \colon \gal(K) \to B, \beta \colon H\to B)$
are said to be \textbf{equivalent} if there exist isomorphisms
$i\colon G\to H$ and $j\colon A\to B$ for which the following
diagram commutes.
\[
\xymatrix{%
G\ar[r]^{\alpha} \ar[d]^i
    &A\ar[d]^j
        &\gal(K)\ar@{=}[d]\ar[l]_{\mu}\\
H\ar[r]^{\beta}
    &B
        &\gal(K)\ar[l]_{\nu}
}%
\]
It is evident that any (weak) solution of $(\mu,\alpha)$
corresponds to a (weak) solution of $(\nu,\beta)$ and vice
versa.

Denote by $L$ the fixed field of $\ker(\mu)$ in $K_s$. Then $\mu$
factors as $\mu = \mugag\mu_0$, where $\mu_0\colon \gal(K)\to
\gal(L/K)$ is the restriction map and $\mugag\colon \gal(L/K)\to A$
is an isomorphism. Then the embedding problems $(\mu,\alpha)$ and
$(\mu_0,\mugag^{-1}\alpha)$ are equivalent.
So, from now on, we shall assume that $A = \gal(L/K)$ and $\mu$ is
the restriction map (unless we explicitly specify differently).
\begin{equation}\label{EP:basic}
\xymatrix{%
    &\gal(K)\ar[d]^{\mu}\ar@{.>}[dl]_{\exists\theta?}\\
G\ar[r]^(0.35){\alpha}
    &\gal(L/K)
}%
\end{equation}
Let $\theta\colon \gal(K) \to G$ be a weak solution of
$(\mu,\alpha)$. The fixed field $F$ of $\ker(\theta)$ is called the
\textbf{solution field}. Then, if $\theta$ is a solution, the
embedding problems $(\mu,\alpha)$ and the embedding problem defined
by the restriction map, i.e.\ $(\mu, \res\colon \gal(F/K) \to
\gal(L/K))$, are equivalent.

\subsection{Geometric and Rational Embedding Problems}
We define two kinds of embedding problems for a field $K$ coming from
geometric objects.

\begin{definition}
Let $E$ be a finitely generated regular extension of $K$, let $F/E$
be a Galois extension, and let $L = F\cap K_s$, where $K_s$ is a
separable closure of $K$. Then the restriction map $\alpha \colon
\gal(F/E)\to \gal(L/K)$ is surjective, since $E\cap K_s = K$.
Therefore
\begin{equation}\label{eq:geometricep}
(\mu\colon \gal(K) \to \gal(L/K), \alpha\colon \gal(F/E)\to
\gal(L/K))
\end{equation}
is an embedding problem for $K$. We call such an embedding problem
\textbf{geometric embedding problem}.

If $E = K(\bft) = K(t_1,\ldots , t_e)$ is a field of rational
functions over $K$, then we call
\begin{equation}\label{eq:rationalep}
(\mu\colon \gal(K) \to \gal(L/K), \alpha\colon \gal(F/K(\bft))\to
\gal(L/K))
\end{equation}
\textbf{rational embedding problem}.
\end{definition}

We can consider only geometric embedding problems because of the
following lemma.

\begin{lemma}
Every finite embedding problem is equivalent to a geometric
embedding problem.
\end{lemma}

\begin{proof}
It follows from \cite{FriedJarden2005}, Lemma~11.6.1.
\end{proof}

The following lemma shows that one can replace the tuple $\bft$ with
a single transcendental element.

\begin{lemma}\label{lem:regularsolvablewithonet}
Let $K$ be an infinite field, $(t,\bft)$ an $(e+1)$-tuple of
variables, and \eqref{eq:rationalep} a rational embedding problem.
Then there exists a rational embedding problem
\[
(\mu_t \colon \gal(K) \to \gal(L/K), \alpha_t\colon \gal(F_t /K(t))
\to \gal(L/K))
\]
which is equivalent to \eqref{eq:rationalep}. Furthermore, there
exists a place $\phi$ of $F$ whose residue field is $F_t$ and such
that $\phi(t_i) = a_i + b_i t$, $a_i,b_i\in K$, and $b_i\neq 0$.
\end{lemma}

\begin{proof}
Let $a_i,b_i\in K$, $b_i\neq 0$ be as given in
\cite{Bary-SorokerCharacterization}, Lemma~4, for $F/L(\bft)$.
Extend the specialization $\bft\mapsto (a_1 + b_1 t, \ldots, a_e +
b_e t)$ to a place of $F$ trivial on $L$ and let $F_t$ be its
residue field. By \cite{Bary-SorokerCharacterization}, Lemma~4, it
follows that $F_t$ is regular over $L$ and that $[F_t:L(t)] =
[F:L(\bft)]$. Then $[F_t:K(t)] = [F:K(\bft)]$, and thus
$\gal(F/K(\bft)) \cong \gal(F_t/K(t))$.
\end{proof}

\subsection{Geometric Solutions}
\label{sec:geometricsolution}%
Consider a geometric embedding problem
$$
(\mu\colon \gal(K)\to \gal(L/K),\alpha\colon \gal(F/E) \to
\gal(L/K))
$$
for a field $K$. Let $\phi$ be a place of $E$ that is unramified in
$F$. We always assume that $\phi$ is trivial on $K$, i.e.\ $\phi(x)
= x$, for $x\in K$. We denote a residue field by bar, e.g., $\Egag$
is the residue field of $E$.

Assume $\phi$ is $K$-rational, i.e.\ $\Egag = K$, and extend it to a
place of  $F$, say $\Phi$. By composing  $\Phi$ with an appropriate
Galois automorphism, we can assume that $\Phi$ is trivial on $L$.

Then $\Phi/\phi$ canonically induces a weak solution $\Phi^*$ of
$(\mu,\alpha)$. The image of $\Phi^*$ is the decomposition group and
$\Phi(\Phi^*(\sigma)(x)) = \sigma (\Phi(x))$ for all $\sigma\in
\gal(K)$, provided $\Phi(x)$ is finite. If we choose a different
extension of $\phi$ to $F$, say $\Psi$, then $\Psi^*$ and $\Phi^*$
differ by an inner automorphism of $\gal(F/E)$. For complete details
see \cite{FriedJarden2005}, Lemma~6.1.4).

\begin{definition}
Let $\theta \colon \gal(K)\to \gal(F/E)$ be a weak solution of
$(\mu,\alpha)$. Then we call it \textbf{geometric} if there exists a
place $\Phi$ of $F$ unramified over $E$ such that $\Egag = K$ and
$\theta = \Phi^*$.
\end{definition}

If we have a commutative diagram 
\[
\xymatrix{
	&\gal(K)\ar[d]_{\mu_1} \ar@/^10pt/[dd]^{\mu_2}\\
H_1\ar@{->>}[r]^{\alpha_1}\ar@{->>}[d]^{\pi}
	&G_1\ar@{->>}[d]\\
H_2\ar@{->>}[r]^{\alpha_2}
	&G_2,
}
\]
then we say that $(\mu_1, \alpha_1)$ \textbf{dominates} $(\mu_2,\alpha_2)$. Note that any (weak) solution $\theta_1\colon \gal(K) \to H_1$ of $(\mu_1,\alpha_1)$ induces a (weak) solution $\theta_2 = \pi \theta_1$ of $(\mu_2,\alpha_2)$.

Geometric solutions are compatible with scalar extensions:
\begin{lemma}\label{lem:geometricdomination}
Consider a geometric embedding problem
$$
(\mu\colon \gal(K)\to \gal(L/K),\alpha\colon \gal(F/E) \to
\gal(L/K)).
$$
Let $M/K$ be a Galois extension with $L\subseteq M$. Then the
geometric embedding problem
\[
(\mu'\colon \gal(K) \to \gal(M/K), \alpha'\colon \gal(FM/E)\to
\gal(M/K)),
\]
where $\alpha'$ and $\mu'$ are the corresponding restriction maps,
dominates the embedding problem $(\mu,\alpha)$ with respect to the
restriction maps. Furthermore, if $\Psi^*$ is a geometric (weak)
solution of $(\mu',\alpha')$, then $(\Psi|_{F})^*$ is a geometric
(weak) solution of $(\mu,\alpha)$.
\end{lemma}

\begin{proof}
As $E/K$ is regular, we have
\begin{eqnarray*}
\gal(FN/E) &=&\gal(F/E)\times_{\gal(LE/E)}\gal(ME/E) \\
&\cong & \gal(F/E)\times_{\gal(L/K)} \gal(M/K)
\end{eqnarray*}
and the projection maps coincide with the restriction maps. Thus
$(\mu',\alpha')$ dominates $(\mu,\alpha)$. Now let $\Psi^*$ be a
geometric (weak) solution of $(\mu',\alpha')$. For $\Phi =
\Psi|_{F}$, we have that $\Phi$ is unramified over $E$ and
$\res_{FM,F}\circ \Psi^* = \Phi^*$, as needed.
\end{proof}

\section{Pseudo Algebraically Closed Extensions and Double Embedding
Problems}

\subsection{Basic Properties}
 The following proposition gives several equivalent definitions of
PAC extensions in terms of polynomials and places, including a
reduction to plane curves. A proof of that proposition essentially
appears in \cite{JardenRazon1994}. Nevertheless, for the sake of
completeness, we give here a formal proof.

\begin{proposition}\label{prop:DefinitionPACextension}
The following conditions are equivalent for a field extension
$K/K_0$.
\begin{label1}
\item \label{con pac:DefinitionPACextension}%
$K/K_0$ is PAC.

\item \label{con poly r:DefinitionPACextension}%
For every $e\geq 1$ and every absolutely irreducible polynomial $f(\bfT,X)\in
K[T_1,\ldots, T_e,X]$ that is separable in $X$, and nonzero
$r(\bfT)\in K[\bfT]$ there exists $(\bfa,b)\in K_0^e\times K$ for
which $r(\bfa)\neq 0$ and $f(\bfa, b) = 0$.

\item \label{con poly:DefinitionPACextension}%
For every absolutely irreducible polynomial $f(T,X)\in K[T,X]$ that
is separable in $X$, and nonzero $r(T)\in K(T)$ there exists
$(a,b)\in K_0\times K$ for which $r(a)\neq 0$ and $f(a,b) = 0$.

\item\label{con a:DefinitionPACextension}
For every finitely generated regular extension $E/K$ with
separating transcendence basis $\bft = (t_1, \ldots , t_e)$ and
every nonzero $r(\bft)\in K(\bft)$, there exists a $K$-rational
place $\phi$ of $E$ unramified over $K(\bft)$ such that
$\overline{K_0(\bft)} = K_0$, $\bfa = \phi(\bft)$ is finite, and
$r(\bfa)\neq 0,\infty$.

\item\label{con b:DefinitionPACextension}%
For every finitely generated regular extension $E/K$ with
separating transcendence basis $t$ and every nonzero $r(t)\in K[t]$
there exists a $K$-rational place $\phi$ of $E$ unramified over
$K(t)$ such that $\overline{K_0(t)} = K_0$, $a = \phi(t)\neq
\infty$, and $r(a)\neq0,\infty$.
\end{label1}
\end{proposition}

\begin{proof}
The proof of \cite{JardenRazon1994}, Lemma~1.3, gives the
equivalence between \eqref{con pac:DefinitionPACextension},
\eqref{con poly r:DefinitionPACextension}, and \eqref{con
poly:DefinitionPACextension}. Obviously \eqref{con
a:DefinitionPACextension} implies \eqref{con
b:DefinitionPACextension}, so it suffices to prove that \eqref{con
poly r:DefinitionPACextension} implies \eqref{con
a:DefinitionPACextension} and that \eqref{con
b:DefinitionPACextension} implies \eqref{con
poly:DefinitionPACextension}.

\eqref{con poly r:DefinitionPACextension} $\Rightarrow$ \eqref{con
a:DefinitionPACextension}: Let $x\in E/K(\bft)$ be integral over
$K[\bft]$ such that $E = K(\bft,x)$. Let $f(\bfT,X)\in K[\bfT,X]$ be
the absolutely irreducible polynomial which is monic and separable
in $X$ and for which $f(\bft,x)=0$. Let $0\neq g(\bft)\in K[\bft]$
be the discriminant of $f(\bfT,X)$ as a polynomial in $X$. We have
$(\bfa,b) \in K_0^e\times K$ such that $f(\bfa,b)=0$ and
$g(\bfa)r(\bfa)\neq 0,\infty$. Extend the specialization
$\bft\mapsto\bfa$ to a $K$-rational place $\phi$ of $E$ with the
following properties to conclude the implication: (1) $\phi(x) =
b\neq \infty$ (this is possible since $x$ is integral over
$K[\bft]$); (2) $\overline{K_0(\bft)} = K_0$ and $\Egag = K(b)=K$
(\cite{FriedJarden2005}, Lemma~2.2.7,); (3) $\phi$ is unramified
over $K(\bft)$ (\cite{FriedJarden2005}, Lemma~6.1.8).

\eqref{con b:DefinitionPACextension} $\Rightarrow$ \eqref{con
poly:DefinitionPACextension}: Let $f(T,X) = \sum_{k=0}^n a_k(T) X^k$
and $r(T)$ be as in \eqref{con poly:DefinitionPACextension}. Set
$r'(T) = r(T) a_n(T)$. Let $t$ be a transcendental element and let
$x\in \widetilde{K(t)}$ be such that $f(t,x) = 0$. Let $E = K(t,x)$.
Then $E$ is regular over $K$ and separable over $K(t)$. Applying
\eqref{con b:DefinitionPACextension} to $E$ and $r'(t)$ we get a
$K$-rational place $\phi$ of $E$ satisfying the following
properties. (1) $a=\phi(t)\in K_0$ which implies that $b=\phi(x)$ is
finite, since $\phi(f(t,x))=0$ and $f(a,X)$ has a nonzero leading
coefficient; (2) $\Egag=K$, which concludes the proof since $b\in
\Egag=K$.
\end{proof}

\subsection{Geometric Solutions and PAC fields}
The following result characterizes when a solution is geometric in
terms of a rational place of some regular extension. This sharpens
earlier works of Roquette on PAC Hilbertian fields
\cite{FriedJarden2005}, Corollary~27.3.3, and of Fried-Haran-Jarden
on Frobenius fields \cite{FriedJarden2005}, Proposition~24.1.4. In the case where $L=K$, it was also proved by D\`ebes in his work on the Beckmann-Black problem \cite{Debes}.

\begin{proposition}\label{prop: characterization of solutions}
Let $K$ be a field and consider a geometric embedding problem
$$
(\mu\colon \gal(K)\to\gal(L/K),\alpha\colon \gal(F/E)\to
\gal(L/K))
$$
for $K$. Let $\theta\colon \gal(K)\to \gal(F/E)$ be a
weak solution. Then there exists a finite separable extension
$\Ehat/E$ such that $\Ehat/K$ is regular and for every place $\phi$
of $E/K$ that is unramified in $F$ the following two conditions are
equivalent.
\begin{enumerate}\renewcommand{\labelenumi}{(\arabic{enumi})}
\item \label{enu:characterization of solutions a}
$\phi$ extends to a place $\Phi$ of $F$ such that $\Egag= K$ and $\Phi^* = \theta$.
\item \label{enu:characterization of solutions b}
$\phi$ extends to a $K$-rational place of $\Ehat$.
\end{enumerate}
\end{proposition}

\begin{proof}
First we consider the special case when $\gal(F/E)\cong \gal(L/K)$.
Then there is a unique solution of $(\mu,\alpha)$, namely $\theta =
\alpha^{-1}\mu$. Let $\Phi$ be an extension of $\phi$ to $F$ and
take $\Ehat = E$. It is trivial that \eqref{enu:characterization of
solutions a} implies \eqref{enu:characterization of solutions b}.
Assume \eqref{enu:characterization of solutions b}. Then $\Phi^*$ is
defined, and from the uniqueness, $\Phi^* = \theta$.

Next we prove the general case. Let $M/K$ be a Galois extension such
that $\gal(M)= \ker(\theta)$ (in particular, $L\subseteq M$) and let
$\Fhat = FM$.
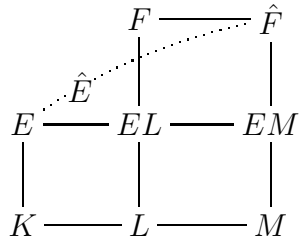
\[
\xymatrix{%
    & F\ar@{-}[r]
    & \Fhat\\
E\ar@{-}[r]\ar@{.}@/^/[urr]|(.25){\mbox{$\Ehat$}}
    & EL\ar@{-}[r]\ar@{-}[u]
        & EM\ar@{-}[u]\\
K\ar@{-}[r]\ar@{-}[u]
    & L\ar@{-}[r]\ar@{-}[u]
        & M\ar@{-}[u]\\
}%
\]
As $F$ and $M$ are linearly disjoint over $L$, the fields $F$ and
$EM$ are linearly disjoint over $EL$. We have
\[
\gal(\Fhat/E) = \gal(F/E) \times_{\gal(L/K)} \gal(M/K).
\]
Define $\hat\theta \colon \gal(K) \to \gal(\Fhat/E)$ by $\hat\theta
(\sigma) = (\theta(\sigma), \sigma|_{M})$. Let $\Ehat$ denote the
fixed field of $\hat\theta(\gal(K))$ in $\Fhat$. Then $\hat\theta$
is a solution in
\[
\xymatrix{%
    &\gal(K)\ar[d]^{\mu}\ar[dl]_{\hat\theta}\\
\gal(\Fhat/\Ehat)\ar[r]^{\hat\alpha}
    &\gal(L/K).
}%
\]
Here $\alphahat$ is the restriction map. In particular, $\Ehat/K$ is
regular. Also $\ker(\hat\theta) = \ker(\theta)\cap \gal(M) =
\gal(M)$, so $\gal(\Fhat/\Ehat)\cong \gal(M/K)$.

Assume $\phi$ extend to a $K$-rational place $\phihat$ of $\Ehat$.
Extend $\phihat$ $L$-linearly to a place $\Phihat$ of $\Fhat$. Then
$\Phihat/\phihat$ is unramified. Let $\Phi = \Phihat|_{F}$. Then by
the first part $\Phihat^* = \thetahat$.
Lemma~\ref{lem:geometricdomination} then asserts that $\Phi^* =
\theta$.

On the other hand, assume that $\phi$ extends to a place $\Phi$ of
$F$ such that $\Egag = K$ and $\Phi^*=\theta$. Extend $\Phi$
$M$-linearly to a place $\Phihat$ of $\Fhat$. Then, since
$\res_{\Fhat,F}(\Phihat^*)=\Phi^*$ and
$\res_{\Fhat,M}(\Phihat^*)=\res_{K_s,M}$, we have
\[
\Phihat^*(\sigma) = (\Phi^*(\sigma),\sigma|_M) = \thetahat.
\]
Therefore the residue field of $\Ehat$ is also $K$.
\end{proof}

\begin{remark}
In the proof it was shown that $\Ehat\subseteq FM$, where $M$ is the
solution field of $\theta$.
\end{remark}

Proposition~\ref{prop: characterization of solutions} is extremely
useful. We first apply it to PAC fields.

\begin{corollary}\label{cor:PACfield_solutionisgeometric}
Every weak solution of a finite geometric embedding problem for a
PAC field $K$ is geometric.
\end{corollary}

\begin{proof}
Let $ (\mu\colon \gal(K)\to\gal(L/K),\alpha\colon \gal(F/E)\to
\gal(L/K))$ be a finite geometric embedding problem for $K$ and
$\theta$ a solution. Let $\Ehat$ be the regular extension of $K$
given in Proposition~\ref{prop: characterization of solutions}.
There exists a $K$-rational place $\phihat$ of $\Ehat$. Then by the
proposition we can extend $\phihat|_{E}$ to a place $\Phi$ of $F$
such that $\Phi^* = \theta$.
\end{proof}

\subsection{Double Embedding Problems}
In this section we generalize the notion of embedding problems to
field extensions.

\subsubsection{The Definition of Double Embedding Problems}
\label{sec:DEP} Let $K/K_0$ be a field extension. A \textbf{double
embedding problem (DEP)} for $K/K_0$ consists of two embedding
problems: $(\mu\colon \gal(K) \to G, \alpha\colon H\to G)$ for $K$
and $(\mu_0\colon \gal(K_0) \to G_0, \alpha\colon H_0\to G_0)$ for
$K_0$, which are compatible in the following sense. $H\leq H_0$,
$G\leq G_0$, and if we write $i\colon H\to H_0$ and $j\colon G\to
G_0$ for the inclusion maps and $r$ for the restriction map $\gal(K) \to \gal(K_0)$, then the following diagram commutes.

\begin{equation}
\label{eq:DoubleEP}%
\xymatrix{%
		&&\gal(K)\ar[d]^r\ar@{->>}[ddr]_{\mu}\ar@{.>}[ddll]_{\exists \theta?}\\
		&&\gal(K_0)\ar@{->>}[d]_{\mu_0}\ar@{.>}[dl]_{\exists\theta_0?}
\\
H\ar@{^(->}[r]^{i}\ar@{->>}@/_10pt/[rrr]_{\alpha}
	&H_0\ar@{->>}[r]^{\alpha_0}
		&G_0
			&G\ar@{_(->}[l]_{j}
}
\end{equation}

Given a DEP for $K/K_0$, we refer to the corresponding embedding
problem for $K$ (resp.\ $K_0$) as the \textbf{lower} (resp.\ the
\textbf{upper}) embedding problem. We call a DEP \textbf{finite} if
the upper (and hence also the lower) embedding problem is finite.

A \textbf{weak solution} of a DEP \eqref{eq:DoubleEP} is a weak
solution $\theta_0$ of the upper embedding problem which restricts
to a weak solution $\theta$ of the lower embedding problem via the
restriction map $r\colon\gal(K)\to \gal(K_0)$. In case $K/K_0$ is a
separable algebraic extension, the restriction map is the inclusion
map, and hence the condition on $\theta_0$ reduces to
$\theta_0(\gal(K))\leq H$. To emphasize the existence of $\theta$,
we usually regard a weak solution of a DEP as a pair
$(\theta,\theta_0)$ (where $\theta$ is the restriction of $\theta_0$
to $\gal(K)$).

\subsubsection{Rational Double Embedding Problems} Consider
a double embedding problem \eqref{eq:DoubleEP} and let $L_0$ and $L$
be the fixed fields of the kernels of $\mu_0$ and $\mu$,
respectively. Then we have isomorphisms $\mugag_0\colon\gal(L_0/K_0)
\to G_0$ and $\mugag\colon\gal(L/K) \to G$. Hence (as in the case of
embedding problems) replacing $G_0$ and $G$ with $\gal(L_0/K_0)$ and
$\gal(L/K)$ (and replacing correspondingly all the maps) gives us an
equivalent DEP. The compatibility condition is realized as $L = L_0
K$.

In the context of this work we are mainly interested in double
embedding problems which satisfy some rationality condition.

\begin{definition}
If in a double embedding problem the upper embedding problem is
rational, we say that the double embedding problem is
\textbf{rational}.
\end{definition}

\begin{lemma}
If \eqref{eq:DoubleEP} is a rational DEP, then the upper embedding
problem is rational and the lower embedding problem is geometric.

Moreover, we can take $e=1$, i.e.\ $\bft = t$ -- a
transcendental element.
\end{lemma}

\begin{proof}
Let \eqref{eq:DoubleEP} be a rational double embedding problem. By
definition, it means that $G_0 = \gal(L_0/K_0)$, there exists a
regular extension $F_0$ of $L_0$ and a separating transcendence
basis $\bft\in F_0$ such that $F_0/K_0(\bft)$ is Galois with Galois
group $H_0 = \gal(F_0/K_0(\bft))$, and $\alpha_0\colon H_0 \to G_0$
is the restriction map. By Lemma~\ref{lem:regularsolvablewithonet}
we may assume that $\bft=t$.

Let $F=F_0K$ and $L = L_0K$. The compatibility condition implies
that $H$ embeds into $\gal(F_0/K_0(\bft))$ (via $i$) as a subgroup
of $\gal(F_0/ L_0 \cap K (\bft)) \cong \gal(F/K(\bft))$. Let
$E\subseteq F$ be the fixed field of $H$, i.e., $H = \gal(F/E)$.
Under this embedding, $\alpha\colon \gal(F/K(\bft))\to \gal(L/K)$ is
the restriction map. Therefore $\alpha(H) = \gal(L/K)$ implies that
$E\cap L = K$, and hence $E$ is regular over $K$.

\begin{equation}
\label{eq:RationalDoubleEP}%
\xymatrix{
		&&\gal(K)\ar[d]^r\ar@{->>}[ddr]^{\mu}\\
		&&\gal(K_0) \ar@{->>}[d]^{\mu_0} \\
\gal(F/E) \ar@{^(->}[r]^-{i}\ar@{->>}@/_15pt/[rrr]_{\alpha}
	&\gal(F_0/K_0(t))\ar@{->>}[r]^{\alpha_0}
		&\gal(L_0/K_0)
			&\gal(L/K).\ar@{_(->}[l]_-{j}
}
\end{equation}

Consequently the lower embedding problem is geometric.
\end{proof}

\begin{remark}
The converse of the above lemma is also valid, that is to say,
assume we have a rational double embedding problem as in
\eqref{eq:RationalDoubleEP}, i.e.\ a finitely generated regular
extension $E/K$, a separating transcendence basis $\bft$ for $E/K$,
and a finite Galois extension $F_0/K_0(\bft)$ such that $E\subseteq
F$, where $F=F_0K$. Then all the restriction maps in
\eqref{eq:RationalDoubleEP} are surjective, and hence
\eqref{eq:RationalDoubleEP} defines a finite double embedding
problem.
\end{remark}

\subsubsection{Geometric Solutions of Double Embedding Problems}
First recall that a weak solution $\theta$ of an embedding problem
\[
(\mu\colon \gal(K)\to \gal(L/K), \alpha\colon \gal(F/E)\to
\gal(L/K))
\]
is geometric if $\theta=\Phi^*$, for some place $\Phi$ of $F$ that
is unramified over $E$ and under which the residue field of $E$ is
$K$. Then we call a weak solution $(\theta,\theta_0)$ of
\eqref{eq:RationalDoubleEP} \textbf{geometric} if
$(\theta,\theta_0)=(\Phi^*,\Phi_0^*)$, where $\Phi^*$ is a geometric
solution of the lower embedding problem and $\Phi_0 = \Phi|_{F_0}$.

Note that since $\Phi_0^*$ is a solution of the upper embedding
problem, the residue field of $K_0(\bft)$ is $K_0$. In particular,
if $\Phi(\bft)$ is finite, then $\Phi(\bft) \in K_0^e$. Also note
that for a place of $F$ that is unramified over $E$ and such that
$\Egag= K$ and $\overline{K_0(\bft)}=K_0$, the pair
$(\Phi^*,\Phi_0^*)$ is indeed a weak solution of
\eqref{eq:RationalDoubleEP}, since $\Phi_0^* =
\res_{K_s,K_{0s}}\Phi^*$.

\section{The Lifting Property}
In this section we formulate and prove the lifting property. First
we reduce the discussion to separable algebraic extensions by
showing that if $K/K_0$ is PAC, then $K\cap K_{0s}/K_0$ is PAC and
$\gal(K)\cong \gal(K\cap K_{0s})$ via the restriction map. Then we
characterize separable algebraic PAC extensions in terms of
geometric solutions of double embedding problems. From this
characterization we establish the lifting property. Finally we prove
a strong (but complicated) version of the lifting problem to PAC
extensions of finitely generated fields.

\subsection{Reduction to Separable Algebraic Extensions}
In \cite{JardenRazon1994}, Corollary~1.5, Jarden and Razon show

\begin{lemma}[Jarden-Razon] If $K/K_0$ is PAC, then so is $K\cap
K_{0s}/K_0$.
\end{lemma}

Moreover, we have

\begin{theorem}\label{thm:nonalgebraicPAC}
Let $K/K_0$ be a PAC extension. Then $K\cap K_{0s}/K_0$ is PAC
and the restriction map $\gal(K)\to \gal(K\cap K_{0s})$ is an
isomorphism.
\end{theorem}

\begin{proof}
It suffices to show that $K_s =  K_{0s} K$. Let $L/K$ be a finite
Galois extension with Galois group $G$ of order $n$. Embed $G$ into
the symmetric group $S_n$. Let $F_0/K_0(\bft)$ be a regular
realization of $S_n$ with $F_0$ algebraically independent from $K$
over $K_0$ (\cite{Lang2002}, p.~272, Example~4). Then $F = F_0 K$ is
regular over $K$ and $\gal(F/K(\bft)) \cong S_n$. Furthermore
\begin{eqnarray*}
\gal(FL/K(\bft)) &=& \gal(FL/L(\bft))\times \gal(FL/F)\\
    &\cong& \gal(F/K(\bft))\times \gal(L/K) \cong S_n \times G.
\end{eqnarray*}
Let $E$ be the fixed field of the subgroup $\Delta=\{(g,g)\mid g\in
G\}$ in $FL$, i.e., $\gal(FL/E) \cong \Delta$. By Galois
correspondence,  $S_n \Delta = S_n\times G$ implies that $E\cap L =
K$ and $1 = \Delta\cap S_n = G\cap \Delta$ implies that $FL = EL =
FE$. In particular, $E/K$ is regular.
\[
\xymatrix{%
F\ar@{-}[r]
    &FL\\
K(\bft)\ar@{-}[u]\ar@{-}[r]\ar@{-}[ur]|{E}
    &L(\bft)\ar@{-}[u]\\
K\ar@{-}[u]\ar@{-}[r]
    &L\ar@{-}[u].
}%
\]

As $K/K_0$ is PAC, there is a $K$-rational place $\phi$ of $E$ such
that $\overline{K_0(\bft)} = K_0$, $\bfa=\phi(\bft)$ is finite, and
$h(\bfa)\neq 0$. Extend $\phi$ to a place $\Phi$ of $FL$ which is
trivial on $L$.

Then since $\Phi/\phi$ is unramified, $FL = EL$ implies that
$\overline{FL} = \overline{EL} = L$ (\cite{FriedJarden2005},
Lemma~2.4.8). However, as $F = F_0 K$, it follows that $\Fgag =
\Fgag_0 K$, hence (again by \cite{FriedJarden2005}, Lemma~2.4.8), $L
= \overline{FL} = \overline{FE}= \Fgag = \Fgag_0 K\subseteq
K_{0s}K$, as needed.
\end{proof}

\begin{remark}
The above theorem also follows from Razon's result
Theorem~\ref{thm:intSplitting}. However we shall prove the converse
implication.
\end{remark}

PAC fields have a nice elementary theory. Since $K$ and $K\cap
K_{0s}$ are PAC fields and since they have isomorphic absolute
Galois groups, they are elementary equivalent under some necessary
condition:

\begin{corollary}
Let $K/K_0$ be a separable PAC extension. Assume that $K$ and $K\cap
K_{0s}$ have the same degree of imperfection. Then $K$ is an elementary
extension of $K\cap K_{0s}$.
\end{corollary}

\begin{proof}
The assertion follows from \cite{FriedJarden2005}, Corollary~20.3.4,
and Theorem~\ref{thm:nonalgebraicPAC}.
\end{proof}

\subsection{Characterization of Separable Algebraic PAC Extensions}

\begin{proposition}
\label{prop:ExistGeoSol}%
Let $K/K_0$ be a separable algebraic field extension. The following
conditions are equivalent:
\begin{label1}
\item
\label{con_a:ExistGeoSol} $K/K_0$ is PAC.
\item
\label{con_b_mid:ExistGeoSol}%
For every finite rational double embedding problem
\eqref{eq:RationalDoubleEP} for $K/K_0$ and every nonzero rational
function $r(\bft)\in K(\bft)$, there exists a  geometric weak
solution $(\Phi^*,\Phi_0^*)$ such that $\bfa=\Phi(\bft)$ is finite
and $r(\bfa)\neq 0,\infty$.
\item
\label{con_b:ExistGeoSol}%
For every finite rational double embedding problem
\eqref{eq:RationalDoubleEP} for $K/K_0$ with $\bft=t$ a
transcendental element there exist infinitely many geometric weak
solutions $(\Phi^*,\Phi_0^*)$.
\end{label1}
\end{proposition}

\begin{proof}
The implication
\eqref{con_a:ExistGeoSol}$\Rightarrow$\eqref{con_b_mid:ExistGeoSol}
follows from Proposition~\ref{prop:DefinitionPACextension} (part
\eqref{con a:DefinitionPACextension}) and the definition of geometric
weak solutions.

Taking $\bft=t$ a transcendental element (instead of a general tuple) in \eqref{con_b:ExistGeoSol} yields 
\eqref{con_b_mid:ExistGeoSol}.

\eqref{con_b:ExistGeoSol}$\Rightarrow$\eqref{con_a:ExistGeoSol}:
We apply Proposition~\ref{prop:DefinitionPACextension} and show
that \eqref{con b:DefinitionPACextension} holds. Let $E/K$ be a
regular extension with a separating transcendence basis $t$ and
let $r(t)\in K[t]$ be nonzero. Choose $F_0$ to be a finite
Galois extension of $K_0(t)$ such that $E\subseteq F_0 K$ (such
$F_0$ exists since $K/K_0$ is separable and algebraic). Let $F =
F_0K$, $L = F\cap K_s$, and $L_0 = F_0\cap K_{0s}$. By
assumption there are infinitely many geometric weak solutions
$(\Phi^*,\Phi_0^*)$ of the DEP
\[
\xymatrix{
		&&\gal(K)\ar[d]\ar[ddr]\ar@{.>}[ddll]_{\Phi^*}\\
		&&\gal(K_0) \ar[d] \ar@{.>}[dl]_{\Phi_0^*}\\
\gal(F/E) \ar[r]\ar@/_15pt/[rrr]
	&\gal(F_0/K_0(t))\ar[r]
		&\gal(L_0/K_0)
			&\gal(L/K).\ar[l]
}
\]

Since for only finitely many solutions $\Phi(t)$ is infinite or
$r(\Phi(t))=0,\infty$ we can find a solution such that $\Phi(t)\neq
\infty$ and $r(\Phi(t))\neq 0,\infty$. In particular, $\Egag = K$
and $\overline{K_0(t)}=K_0$, as required in \eqref{con
b:DefinitionPACextension}.
\end{proof}

Let $K/K_0$ be a PAC extension and consider a rational DEP for
$K/K_0$. The following key property -- the lifting property --
asserts that any weak solution of the lower embedding problem can be
lifted to a geometric weak solution of the DEP.

\begin{proposition}[The lifting property]
\label{prop:ExtensionSoltoGeoSol_DEP}%
Let $K/K_0$ be a PAC extension, let \eqref{eq:RationalDoubleEP} be a
rational DEP for $K/K_0$, and let $\theta\colon \gal(K)\to
\gal(F/E)$ be a weak solution of the lower embedding problem in
\eqref{eq:RationalDoubleEP}. Then there exists a geometric weak
solution $(\Phi^*,\Phi_0^*)$ of \eqref{eq:RationalDoubleEP}
such that $\theta = \Phi^*$. \\
Moreover, if $r(\bft)\in K(\bft)$ is nonzero, we can choose $\Phi$
such that $\bfa = \Phi(\bft)\in K_0^e$ and $r(\bfa)\neq 0,\infty$.
\end{proposition}

\begin{proof}
By Proposition~\ref{prop: characterization of solutions} there
exists a finite separable extension $\Ehat/E$ that is regular over
$K$ with the following property. If a place $\phi$ of $E$ that is
unramified in $F$ can be extended to a $K$-rational place of
$\Ehat$, then it can be extended to a place $\Phi$ of $F$ such that
$\Phi^* = \theta$.

By the PACness of $K/K_0$ there exists a $K$-rational place
$\phihat$ of $\Ehat$ such that $\phi = \phihat|_{E}$ is unramified
in $F$, the residue field of $K_0(\bft)$ is $K_0$, $\bfa =
\Phi(\bft)$ is finite, and $r(\bfa)\neq 0,\infty$. If we extend
$\phi$ to the place $\Phi$ of $F$ given above and let
$\Phi_0=\phi|_{F_0}$, then we get that $(\Phi^*,\Phi_0^*)$ is a
geometric weak solution and that $\Phi^*=\theta$.
\end{proof}

The first easy consequence of the lifting property is the
transitivity of PAC extensions. We first prove
Theorem~\ref{thm:transitivity} and then deduce a more general
result.

\begin{proof}[Proof of Theorem~\ref{thm:transitivity}]
Let $K_0\subseteq K_1 \subseteq K_2$ be a tower of separable
algebraic extensions and assume that $K_2/K_1$ and $K_1/K_0$ are PAC
extensions. We need to prove that $K_2/K_0$ is PAC too.

Let
\begin{eqnarray*}
&&((\mu_0\colon \gal(K_0)\to \gal(L_0/K_0),\alpha_0\colon \gal(F_0/K_0(t))\to \gal(L_0/K_0)),\\
&&(\mu_2\colon \gal(K_2)\to \gal(L_2/K_2),\alpha_2\colon \gal(F_2/E)
\to \gal(L_2/K_2)))
\end{eqnarray*}
be a rational finite DEP for $K_2/K_0$. By
Lemma~\ref{prop:ExistGeoSol} it suffices to find a geometric weak
solution to $((\mu_0,\alpha_0),(\mu_2,\alpha_2))$. Set $F_1 = F_0
K_1$, $L_1 = L_0 K_1$. Then, since $K_2/K_1$ is PAC there exists a
weak solution $(\Phi_2^*,\Phi_1^*)$ of the double embedding problem
defined by the lower part of the following commutative diagram.
\begin{equation}
\xymatrix{
    &\gal(K_0)\ar[dr]^{\mu_0}\ar@{.>}[dl]_{\Phi_0^*}\\
\gal(F_0/K_0(t))\ar[rr]^(.4){\alpha_0}
        &&\gal(L_0/K_0)\\
    &\gal(K_1)\ar'[u][uu]\ar[dr]^{\mu_1}\ar@{.>}[dl]_{\Phi_1^*}\\
\gal(F_1/K_1(t))\ar[uu]\ar[rr]^(.4){\alpha_1}
        &&\gal(L_1/K_1)\ar[uu]\\
    &\gal(K_2)\ar'[u][uu]\ar[dr]^{\mu_2}\ar@{.>}[dl]_{\Phi_2^*}\\
\gal(F_2/E)\ar[uu]\ar[rr]^(.4){\alpha_2}
        &&\gal(L_2/K_2).\ar[uu]
}
\end{equation}
Now we lift $\Phi_1^*$ to a geometric weak solution
$(\Phi_1^*,\Phi_0^*)$ of the DEP for $K_1/K_0$ defined by the higher
part of the diagram. This is possible by the lifting property
applied to the PAC extension $K_1/K_0$.

Since $\Phi_0^*|_{\gal(K_2)} = \Phi^*_1|_{\gal(K_2)}=\Phi_2^*$ we
get that $(\Phi_2^*,\Phi_0^*)$ is a geometric weak solution of the
DEP we started from.
\end{proof}

\begin{corollary}\label{cor:transitivityofPACextension}
Let $\kappa$ be an ordinal number and let
\[
K_0 \subseteq K_1 \subseteq K_2 \subseteq \cdots \subseteq K_\kappa
\]
be a tower of separable algebraic extensions. Assume that
$K_{\alpha+1}/K_\alpha$ is PAC for every $\alpha<\kappa$ and that
$K_\alpha = \bigcup_{\beta<\alpha} K_\beta$ for every limit
$\alpha\leq \kappa$. Then $K_\kappa/K_0$ is PAC.
\end{corollary}

\begin{proof}
We apply transfinite induction. Let $\alpha\leq \kappa$. If
$\alpha$ is a successor ordinal, then the assertion follows from
Theorem~\ref{thm:transitivity}.

Let $\alpha$ be a limit ordinal. Consider a rational DEP
\begin{equation}\label{eq:transfinitetransitive}
\xymatrix{%
		&&\gal(K_\alpha)\ar[d]\ar[ddr]\\
		&&\gal(K_0)\ar[d]\\
\gal(F_\alpha)\ar[r]\ar@/_15pt/[rrr]
	&\gal(F_0/K_0(\bft))\ar[r]
		&\gal(L_0/K_0)
			&\gal(L_\alpha/K_\alpha)\ar[l]
}%
\end{equation}
Here $F_\alpha = F_0 K_\alpha$ and $L_\alpha = L_0K_\alpha$. Now
since all the extensions are finite, there exists $\beta<\alpha$
such that $\gal(F_\alpha/K_\alpha) \cong \gal(F_\beta/K_\beta)$ and
$\gal(L_\alpha/K_\alpha) \cong \gal(L_\beta/K_\beta)$ (via the
corresponding restriction maps), where $F_\beta = F_0 K_\beta$ and
$L_\beta=L_0K_\beta$.

Induction gives a weak solution of the double embedding problem
\eqref{eq:transfinitetransitive} with $\beta$ replacing $\alpha$.
This weak solution induces a weak solution of
\eqref{eq:transfinitetransitive} via the above isomorphisms.
\end{proof}

\subsection{Strong Lifting Property for PAC Extensions of Finitely
Generated Fields} Let $K_0$ be a finitely generated field (over its
prime field). In this section we prove a strong lifting property for
PAC extensions $K/K_0$. The additional ingredient is the Mordell
conjecture for finitely generated fields (now a theorem due to
Faltings in characteristic $0$ \cite{Faltings83} and to
Grauert-Manin in positive characteristic \cite{Samuel66}, page~107).

The following lemma is based on the Mordell conjecture.

\begin{lemma}[{\cite{JardenRazon1994}, Proposition~5.4}]\label{lem:Mordell}
Let $K_0$ be a finitely generated infinite field, $f\in K_0[T,X]$
an absolutely irreducible polynomial which is separable in $X$,
$g\in K_0[T,X]$ an irreducible polynomial which is separable in $X$,
and $0\ne r\in K_0[T]$. Then there exist a finite purely inseparable
extension $K_0'$ of $K_0$, a nonconstant rational function $q\in
K_0'(T)$, and a finite subset $B$ of $K_0'$ such that $f(q(T),X)$ is
absolutely irreducible, $g(q(a),X)$ is irreducible in $K_0'[X]$, and
$r(q(a))\ne 0$ for any $a\in K_0'\smallsetminus B$.
\end{lemma}

Let $K/K_0$ be an extension. Consider a rational double embedding
problem \eqref{eq:RationalDoubleEP} for $K/K_0$ (with $\bft=t$). For
any subextension $K_1$ of $K/K_0$ we have a corresponding rational
double embedding problem. Namely

\begin{equation}
\label{eq:CorresponingRegularDoubleEP}%
\xymatrix{%
    &&\gal(K)\ar[d]\ar[ddr]^{\mu}\\
    &&\gal(K_1)\ar[d]^{\mu_1}\\
\gal(F/E)\ar@/_10pt/[rrr]_{\alpha}\ar[r]
	&\gal(F_1/K_1(t))\ar[r]^{\alpha_1}
		&\gal(L_1/K_1)
    	    		&\gal(L/K),\ar[l]
}%
\end{equation}
where $F_1 = F_0 K_1$, $L_1=L_0K_1$, and $\mu_1$ and $\alpha_1$ are
the restriction maps.

Assume that  $K_1'/K_1$ is a purely inseparable extension.
Then the double embedding problem
\eqref{eq:CorresponingRegularDoubleEP} remains the same if we
replace all fields by their compositum with $K_1'$.

\begin{proposition}[Strong lifting property]
Let $K$ be a PAC extension of a finitely generated field $K_0$. Let
\[
\calE(K)  = (\mu\colon \gal(K) \to \gal(L/K), \alpha \colon
\gal(F/E)\to \gal(L/K))
\]
be a geometric embedding problem and let $\theta\colon \gal(K) \to
\gal(F/E)$ be a weak solution of $\calE(K)$. Then there exist a
finite subextension $K_1/K_0$ and a finite purely inseparable
extension $K_1'/K_1$ satisfying the following properties.
\begin{enumerate}
\item For any rational double embedding problem \eqref{eq:CorresponingRegularDoubleEP} for $K/K_1$
whose lower embedding problem is $\calE(K)$, we can lift $\theta$ to
a weak solution $(\theta,\theta_1)$ of the double embedding problem
\eqref{eq:CorresponingRegularDoubleEP} in such a way that $\theta_1$
is surjective.
\item
The solution $(\theta,\theta_1)$ is a geometric solution of the
double embedding problem that we get from
\eqref{eq:CorresponingRegularDoubleEP} by replacing all fields with
their compositum with $K_1'$.
\end{enumerate}
\end{proposition}

\begin{proof}
By Proposition~\ref{prop: characterization of solutions} there
exists a finite separable $\Ehat/E$ that is regular over $K$ such
that a $K$-place $\phi$ of $E$ that is unramified in $F$ satisfies
$\phi^*=\theta$ if and only if $\phi$ extends to a $K$-rational
place of $\Ehat$. Let $f(t,X)\in K[t,X]$ be an absolutely
irreducible polynomial whose root $x$ generates $\Ehat/K(t)$, i.e.\
$\Ehat = K(t,x)$. Let $M$ be the fixed field of $\ker(\theta)$ in
$K_s$. Then $M/K$ is a finite Galois extension. Let $h(X)\in K[X]$
be a Galois irreducible polynomial whose root generates $M/K$.

Let $K_1$ be a finite subextension of $K/K_0$ that contains the
coefficients of $f$ and $h$ and such that $h$ is Galois over it. Let
$M_1$ be the splitting field of $h$ over $K_1$ and let $L_1$, $F_1$
be as in the corresponding rational double embedding problem
\eqref{eq:CorresponingRegularDoubleEP}. Then $\gal(M/K)\cong
\gal(M_1/K_1)$, and thus also $\gal(L/K)\cong \gal(L_1/K_1)$.

\[
\xymatrix@R=10pt{
            &&&F\\
    &E \ar@{-}[r]
        &EL\ar@{-}[rr]\ar@{-}[ru]
                &&EM\\
    &K(t) \ar@{-}[r]\ar@{-}[u]
        &L(t)\ar@{-}[rr]\ar@{-}[u]
                &&M(t)\ar@{-}[u]\\
            &&&F_1\ar@{-}'[u]'[uu][uuu]\\
K_0(t)\ar@{-}[r]
    &K_1(t) \ar@{-}[r]\ar@{-}[uu]
        &L_1(t)\ar@{-}[rr]\ar@{-}[uu]\ar@{-}[ur]
                &&M_1(t)\ar@{-}[uu]
    }
\]

Let $g(t,X)\in K_1[T,X]$ be an irreducible polynomial whose root
generates $F_1/K_1(t)$. Choose $r(t)\in K_1(t)$ such that $r(a)\neq
0$ implies that the prime $(t-a)$ is unramified in $F_1$ and that
the leading coefficients of $f(t,X)$ and $g(t,X)$ do not vanish at
$a$. Let $K_1'/K_1$ be the purely inseparable extension, $B\subseteq
K_1'$ the finite subset, and $q\in K_1'(T)$ the nonconstant rational
function that Lemma~\ref{lem:Mordell} gives for $K_1$, $g$, $f$, and
$r$. Let $K' = KK_1'$.

Since $K'/K_1'$ is PAC (\cite{JardenRazon1994}, Corollary~2.5) there
exist $a\in K_1'\smallsetminus B$ and $b\in K'$ for which
$f(q(a),b)=0$ (Proposition~\ref{prop:DefinitionPACextension}).
Extend $t\mapsto q(a)$ to a $K'$-rational place $\phihat$ of $\Ehat
K_1'$. Then $\phi=\phihat|_{E K_1'}$ is unramified in $FK_1'$ (since
$r(q(a))\neq 0$).

By Proposition~\ref{prop: characterization of solutions} $\phi$
extends to a place $\Phi$ of $FK_1'$ such that $\Phi^* = \theta$ and
write $\Phi_1 = \Phi|_{F_1}$. Then $(\Phi^*,\Phi_1^*)$ is a
geometric weak solution of the DEP $((\mu,\alpha),(\mu_1,\alpha_1))$
that we get from \eqref{eq:CorresponingRegularDoubleEP} by replacing
all fields with their compositum with $K_1'$.

Moreover, since $F_1K_1'/K_1'(t)$ is generated by $g(t,X)$ and
$g(q(a),X)$ is irreducible, we get that $\Phi_1^*$ is surjective.
This proves (b). Now assertion (a) follows since $(\Phi^*,\Phi_1^*)$
is a (not necessarily geometric) solution of
$((\mu,\alpha),(\mu_1,\alpha_1))$.
\end{proof}

\section{The Galois Closure of PAC Extensions}
This section proves the main result of the paper,
Theorem~\ref{thm:intmineGalois}, that says the Galois closure of a
proper separable algebraic PAC extension is the separable closure.
The proof uses the lifting property and some properties of realizing
wreath products in fields. We start by recalling the latter and then
prove the theorem.

\subsection{Wreath Products in Fields}
Let $A$ and $G$ be finite groups. The \textbf{wreath product} $A\wr
G$ is defined to be the semidirect product $A^G \rtimes G$, where
$G$ acts on $A^G$ by translation. More precisely, $A^G = \{ f\colon
G\to A\}$ and
\[
\qquad f^\sigma(\tau) = f(\sigma \tau)
\]
for all $f\in A^G$ and $\sigma,\tau \in G$. Then each element of
$A\wr G$ can be written uniquely as $f \sigma$, $f\in A^G$ and
$\sigma\in G$ and the multiplication is given by
\[
(f\sigma) (g\tau) = f g^{\sigma^{-1}} \sigma\tau.
\]
The wreath product is equipped with the quotient map $\alpha \colon
A\wr G \to G$ defined by $\alpha(f\sigma) =  \sigma$.
The following lemmas describe two basic facts on embedding problems
with wreath products.
\begin{lemma}\label{lem:criterionforpropernesswreath}
Let $(\nu\colon \Gamma\to G, \alpha\colon A\wr G \to G)$ be a finite
embedding problem for a profinite group $\Gamma$, assume $G\neq 1$,
and let $\theta\colon \Gamma\to A\wr G$ be a weak solution. Then
the only subgroup of $A^1$ that is normal in
$\theta(\Gamma)$ is the trivial subgroup.
\end{lemma}

\begin{proof}
For each $\sigma \in G$ choose $\gamma \in \Gamma$ such that
$\nu(\gamma) = \sigma$, and let $f_\sigma = \theta (\gamma)
\sigma^{-1}$, i.e.\ $\theta(\gamma) = f_\sigma \sigma$.

Assume that $B\leq A^1$ is normal in $\theta(\Gamma)$. Let $1\neq
\sigma \in G$. Then since
\[
B = B^{f_\sigma \sigma} \leq (A^1)^{f_\sigma\sigma } = A^\sigma
\]
we have
\[
B \leq A^1 \cap A^\sigma = 1,
\]
and hence $B=1$.
\end{proof}

\begin{lemma} \label{lem:rational_wreath}
Let $K_0 \subseteq L$ be a finite Galois extension with Galois
group $G = \gal(L/K_0)$ and let $\mu\colon \gal(K_0) \to G$ be the
restriction map. Let $A$ be a finite group that is regular over
$K_0$. Then the embedding problem
\[
(\mu \colon \gal(K_0) \to G, \alpha\colon A\wr G \to G)
\]
is rational.
\end{lemma}

\begin{proof}
Since $A$ is regular over $K_0$ there exists an absolutely
irreducible polynomial $f(T,X) \in K_0[T,X]$ that is Galois over
$K(T)$ and $\gal(f(T,X), K(T)) \cong A$.

Choose a basis $c_1, \ldots, c_n$ of $L/K_0$ and let $\bft = (t_1,
\ldots , t_n)$ be an $n$-tuple of variables. By
\cite{Haran1999InventMath}, Lemma~3.1 (with $L_0 = L$ and $G_0 =1$)
there exist a field $\Fhat$ such that
\begin{enumerate}
\item $\Fhat$ is regular over $L_0$,
\item $\gal(\Fhat/K_0(\bft)) \cong A\wr G$ and under this
identification $\alpha \colon A\wr G \to G$ coincides with the
restriction map $\gal(\Fhat/ K_0(\bft)) \to \gal(L_0/K_0)$.
\end{enumerate}
In particular we get that $(\mu,\alpha)$ is rational, as claimed.
\end{proof}

\begin{remark}
One can formulate and prove a much more general result than the
above lemma. This generalization considers a split embedding problem
$A\rtimes G_0\to G_0$ instead of $A$, and gives a rational embedding
problem with the twisted wreath product $A\wr_{G_0} G$ instead of
$A\wr G$.

The proof of this generalization is a bit more technical, but still
it uses only \cite{Haran1999InventMath}, Lemma~3.1. We will not use the
generalization here, so we decided to omit it. For the full version
see \cite{Bary-SorokerPhD}.
\end{remark}

\subsection{Proof of Theorem~\ref{thm:intmineGalois}}
Let $K/K_0$ be a proper separable PAC extension. We need to prove
that the Galois closure of $K/K_0$ is $K_{0s}$.

We break the proof into two steps.

\textbf{Step A.} \textit{If $K/K_0$ is a proper Galois PAC
extension, then $K = K_{0s}$}.\\
Let $K_0\subsetneq L\subseteq K$ be a finite Galois extension with a
Galois group $G = \gal(L/K_0)$. Let $N/K$ be a finite Galois
extension with Galois group $B = \gal(N/K)$. It suffices to show
that $B = 1$.

Identify $B$ with a subgroup of $A = S_n$, for some sufficiently
large $n$. Let $\nu \colon \gal(K_0)\to G$ and $\theta \colon
\gal(K) \to B\leq A$ be the restriction maps. Since $A$ is regular
over any field \cite{Lang2002}, p.~272, Example~4, and in particular
over $K_0$, we get that the embedding problem $(\mu,\alpha)$ is
rational (Lemma~\ref{lem:rational_wreath}). Hence, by definition,
the double embedding problem
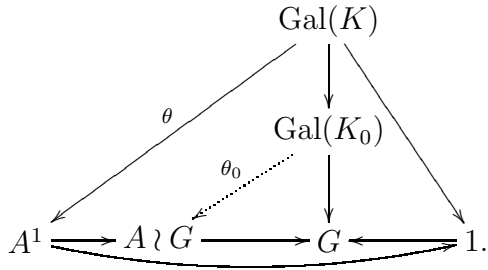
\[
\xymatrix{%
		&&\gal(K)\ar[ddr]\ar[ddll]_{\theta}\ar[d]\\
		&& \gal(K_0)\ar[d]\ar@{.>}[dl]_{\theta_0}\\
A^{1}\ar[r]\ar@/_10pt/[rrr]
	&A\wr G\ar[r]
		&G
			&1.\ar[l]
}%
\]
is rational. By the lifting property
(Proposition~\ref{prop:ExtensionSoltoGeoSol_DEP}) we can extend the
weak solution $\theta$ of the lower embedding problem to a weak
solution $(\theta,\theta_0)$ of the double embedding problem. Now
since $\gal(K)\normal \gal(K_0)$ we have $B = \theta_0(\gal(K))
\normal \theta_0(\gal(K_0))$. Thus $B=1$
(Lemma~\ref{lem:criterionforpropernesswreath}), as needed.

\textbf{Step B.} \textit{The general case}. Let $M$ be the Galois
closure of $K/K_0$. We need to show that $M = K_{0s}$. By
Theorem~\ref{thm:intSplitting} (which was proved in \cite{Razon2000}
and will be reproved below) there exists $M_0/K_0$ that is linearly
disjoint from $K$ over $K_0$ and such that $M = KM_0$.
\[
\xymatrix@R=10pt@C=10pt{%
K\ar@{-}[r]& M\\
K_0\ar@{-}[u]\ar@{-}[r]&M_0\ar@{-}[u]
}%
\]
In particular $M/M_0$ is a proper Galois extension. By
\cite{JardenRazon1994}, Lemma~2.1, $M/M_0$ is PAC. We get from the
first step that $M = K_{0s}$. \hfill \qed

\subsection{PAC fields being non-PAC over any proper subfield}
The following observation follows directly from
Theorem~\ref{thm:intmineGalois}.
\begin{lemma}
Let $N\neq \tilde \bbQ$ be a Galois extension of $\bbQ$. Then $N$ is
a PAC extension of no proper subfield.
\end{lemma}

\begin{proof}
If $K$ is a proper subfield of $N$, then $\bbQ \subseteq K$. In
particular $N/K$ is Galois, and hence Theorem~\ref{thm:intmineGalois}
implies that $N/K$ is not PAC.
\end{proof}

Some Galois extension of $\bbQ$ are known to be PAC as fields.
Hence we get examples of PAC fields which are PAC extensions of no
proper subfield.
\begin{description}
\item[Example i] The Galois hull $\bbQgal[\bfsigma]$ of $\bbQ$ in
$\bbQgal(\bfsigma)$, for almost all $\bfsigma\in \gal(\bbQ)^e$
\cite{FriedJarden2005}, Theorem~18.10.2.
\item[Example ii]
$\bbQ_{\textnormal{tr}}(i)$, where $\bbQ_{\textnormal{tr}}$ is the
maximal real Galois extension of $\bbQ$ and $i^2 = -1$
\cite{Pop1996}.
\item[Example iii]
The compositum $\bbQ_{\textnormal{sym}}$ of all Galois extensions of
$\bbQ$ with a symmetric Galois group \cite{FriedJarden2005},
Theorem~18.10.3.
\end{description}

Over a finite field any infinite algebraic extension is PAC \cite{FriedJarden2005}, Corollary~11.2.4. Thus we get 
\begin{description}
\item[Example iv]
Let $N$ be an infinite extension of a finite field $\bbF_p$ which is not algebraically closed. Then $N$ is a PAC field. However $N$ is Galois over any subfield (since $\gal(\bbF_p)=\hat{\bbZ}$ is abelian). 
Hence, by Theorem~\ref{thm:intmineGalois}, $N$ is a PAC extension of no proper subfield. 
\end{description}

\subsection{Finite PAC Extensions -- Proof of Theorem~\ref{thm:finitePACext}}
Let $K/K_0$ be a finite extension. We need to prove that $K/K_0$ is
PAC if and only if either $K_0$ is real closed and $K$ is its
algebraic closure or $K_0$ is a PAC field and $K/K_0$ is a finite
purely inseparable extension.

Since an algebraically closed field is PAC over any infinite subset
of it we have that indeed $K_0$ is real closed and $K$ is its
algebraic closure implies that $K/K_0$ is PAC. Moreover $[K:K_0] =
2$ (Artin-Schreier Theorem \cite{Lang2002}, VI\S Corollary~9.3).

Let $K_0$ be PAC and $K/K_0$ a finite purely inseparable extension.
Then \cite{JardenRazon1994}, Corollary~2.3, asserts that $K/K_0$ is
PAC.

For the other direction, assume that $K/K_0$ is a finite PAC
extension. Let $K_1$ be the maximal separable extension of $K_0$
contained in $K$. Then $K/K_1$ is purely inseparable
\cite{Lang2002}, V\S6 Proposition~6.6. By \cite{JardenRazon1994},
Corollary~2.3, $K_1/K_0$ is PAC, and in particular $K_1$ is a PAC
field. If $K_1 = K_0$, we are done, since $K/K_0$ is then purely
inseparable.

Assume $K_1\neq K_0$. By Theorem~\ref{thm:intmineGalois}, the Galois
closure $N$ of $K_1/K_0$ is the separable closure. Hence, by
Artin-Schreier Theorem \cite{Lang2002}, VI\S Corollary~9.3, $N$ is,
in fact, algebraically closed and $K_0$ is real closed (recall that
$1<[N:K_0]<\infty$). In particular, the characteristic of $K$ is
$0$, and hence $K_1=K$. \hfill \qed

\section{Descent Features} \label{sec:descent}

\subsection{Proof of Theorem~\ref{thm:intmineDescent}}
Let $K/K_0$ be a PAC extension and $L/K$ a finite Galois extension.
Assume $\gal(L/K) \leq G_0$, where $G_0$ is regular over $K_0$.

We need to find a Galois extension $L_0/K_0$ such that
$\gal(L_0/K_0)\leq G_0$ and $L = L_0 K$.
\[
\xymatrix@R=15pt@C=20pt{%
    &K\ar@{-}[r]^G
        &L\\
K_0\ar@{-}[r]
    &L_0\cap K\ar@{-}[r]^(.6){G}\ar@{-}[u]
        &L_0\ar@{-}[u]
}%
\]

The restriction map $\theta\colon \gal(K)\to G$ is a solution of the
lower embedding problem of the rational double embedding problem
\begin{equation}
\label{eq thm realizing groups}%
\xymatrix{%
		&&\gal(K)\ar[ddr]\ar[ddll]_{\theta}\ar[d]\\
		&&\gal(K_0)\ar[d]\ar@{.>}[dl]_{\theta_0}\\
G\ar[r]\ar@/_10pt/[rrr]
	&G_0\ar[r]
		&1
			&1.\ar[l]    \\
}%
\end{equation}
It extends to a geometric weak solution $(\theta,\theta_0)$ of
\eqref{eq thm realizing groups} by the lifting property
(Proposition~\ref{prop:ExtensionSoltoGeoSol_DEP}). Let $L_0$ be the
fixed field of $\ker (\theta_0)$. Then $\gal(L_0/K_0) =
\theta_0(\gal(K_0))\leq \gal(L_0/K_0)$. Since $\theta_0(\gal(K)) =
\theta(\gal(K)) = G$ we get that $L = L_0K$. \hfill \qed

\subsection{Corollaries of Theorem~\ref{thm:intmineDescent}}
If the group $G$ in Theorem~\ref{thm:intmineDescent} is regular over
$K_0$ we can take $G = G_0$:

\begin{corollary}\label{cor:decentofregularlyrealizable}
Let $K/K_0$ be a PAC extension. Let $G$ be a finite Galois group
over $K$ that is regular over $K_0$. Then $G$ occurs as a Galois
group over $K_0$.
\end{corollary}

Since every abelian group is regular over any field (see e.g.\
\cite{FriedJarden2005}), Proposition~16.3.5, we get the following
\begin{corollary}
Let $K/K_0$ be a PAC extension. Then $K^{\rm ab} = K_0^{\rm ab} K$.
\end{corollary}


From the fact that the symmetric group is regular over any field
Theorem~\ref{thm:intmineDescent} gives a new proof of
Theorem~\ref{thm:intSplitting}. This new proof provides an insight
into Razon's original technical proof of
Theorem~\ref{thm:intSplitting} in \cite{Razon2000}.

\begin{proof}[Proof of Theorem~\ref{thm:intSplitting}]
Let $K/K_0$ be a PAC extension and let $L/K$ be a separable
extension. We need to find a separable $L_0/K_0$ that is
linearly disjoint from $K$ over $K_0$ such that $L = L_0K$.

First assume that $[L:K]$ is finite. Let $M$ be the Galois closure
of $L/K$, $G=\gal(M/K)$, $G' = \gal(M/L)$.
The action of $G$ on the cosets $\Sigma=G/G'$ admits an embedding
$i\colon G\to S_\Sigma$.

As $S_\Sigma$ is regular over $K_0$ (\cite{Lang2002}, p.~272,
Example~4), Theorem~\ref{thm:intmineDescent} gives a Galois
extension $M_0/K_0$ with Galois group $H = \gal(M_0/K_0)$ such that
$H\leq S_\Sigma$ and $G\leq H$ (since $M = M_0 K$). Then $H$ is
transitive, since $G$ is. Thus $(H:H') = |\Sigma|= [L:K]$, where
$H'$ is the stabilizer in $H$ of the coset $G'\in \Sigma$. Also, as
the subgroup $G'\leq G$ is the stabilizer in $G$ of the coset $G'\in
\Sigma$, it follows that $H'\cap G=G'$.

Let $L_0\subseteq M_0$ be the corresponding fixed field of $H'$
(i.e.\ $\gal(M_0/L_0) = H'$). So by the Galois correspondence
$\gal(L) = \gal(L_0)\cap \gal(K) = \gal(L_0 K)$, hence $L = L_0 K$.
In addition, $L_0$ is linearly disjoint from $K$, since
$[L_0:K_0]=(H:H') = [L:K]$, as needed.

The case where $L/K$ is an infinite extension follows from Zorn's
Lemma. The main point is that for a tower of algebraic extensions
$L_1\subseteq L_2\subseteq L_3$, $L_3/L_1$ is separable if and only
if both $L_2/L_1$ and $L_3/L_2$ are. The details can be found in
\cite{Razon2000}.
\end{proof}

\begin{remark}
In the last proof $H'$ was the stabilizer of a point of a subgroup
of $S_n$. This stabilizer is, in general, not normal even if $L/K$
is Galois. That is to say, $L_0/K_0$ need not be Galois, even if
$L/K$ is.
\end{remark}

\section{Fields which are Finite Separable/Galois Extensions of no Proper Subfield}
In this section we prove the generalizations of
Theorems~\ref{thm:intChatzidakis} and \ref{thm:intbottom}. In
particular we settle Problem~18.7.8 of \cite{FriedJarden2005} for
finitely generated infinite fields. Before doing this we need a
technical preparation about Hilbertian fields over subsets and their
relation with PAC extensions.

\subsection{Hilbertian Fields over Subsets}
Let us start by introducing some notation. Let
\[
f_1(T_1,\ldots, T_e ,X_1, \ldots, X_n),\ldots, f_m(T_1,\ldots, T_e
,X_1, \ldots, X_n)\in K[\bfT,\bfX]
\]
be irreducible polynomials and $g(\bfT)\in K[\bfT]$ nonzero. The
corresponding \textbf{Hilbert set} is the set of all irreducible
specializations $\bfT\mapsto \bfa \in K^e$ for $f_1,\ldots, f_m$
under which $g$ does not vanish, i.e.\
\[
H_K(f_1,\ldots, f_m;g) = \{ \bfa\in K^r \mid \forall i\:
f_i(\bfa,\bfX) \mbox{ is irreducible in } K[\bfX]\ \mbox{and} \
g(\bfa)\neq 0\}.
\]
Now $K$ is \textbf{Hilbertian} if any Hilbert set is nonempty
provided that $n=1$ and $f_i=f_i(\bfT,X)$ is separable in $X$ for
each $i$. (Some authors use the terminology `$K$ is separably
Hilbertian'.) A stronger property is that any Hilbert set for $K$ is
nonempty. We call such a field \textbf{s-Hilbertian}.

In case the characteristic of $K$ is zero, these two properties
coincide. If the characteristic of $K$ is positive, there is a
simple criterion for a Hilbertian field to be s-Hilbertian.

\begin{theorem}[Uchida {\cite{FriedJarden2005}, Proposition~12.4.3}]
Let $K$ be a Hilbertian field of characteristic $p>0$. Then $K$ is
s-Hilbertian if and only if $K$ is imperfect.
\end{theorem}
(Recall that $K$ is imperfect if $[K:K^p] >1$.)

\begin{definition}
A field $E$ is said to be \textbf{Hilbertian over} a subset $K$ if
\[
H_E(f_1,\ldots,f_m;g)\cap K^r\neq \emptyset
\]
for any irreducible $f_1,\ldots,f_m\in E[\bfT, X]$ that are
separable in $X$ and any nonzero $g(\bfT)\in E[\bfT]$. If
furthermore $H_E(f_1,\ldots,f_m;g)\cap K^r\neq \emptyset$ for any
irreducible $f_1,\ldots,f_m\in E[\bfT,\bfX]$, then we say that $E$
is \textbf{s-Hilbertian over} $K$.
\end{definition}

Note that a field $K$ is Hilbertian (resp.\ s-Hilbertian) if and
only if it is Hilbertian (resp.\ s-Hilbertian) over itself.

Jarden and Razon prove that if $R$ is a ring with quotient field $K$
and $K$ is a countable Hilbertian field over $R$, then
$K_s(\bfsigma)/R$ and $\Kgal(\bfsigma)$ are PAC for almost all
$\bfsigma\in \gal(K)^e$ \cite{JardenRazon1994}, Proposition~3.1.

A crucial observation for our applications is that the proof of
\cite{JardenRazon1994}, Proposition~3.1, gives the following
stronger statement.

\begin{theorem}
Let $E$ be a countable field that is Hilbertian over a subset $K$.
Then for almost all $\bfsigma\in \gal(E)^e$ the fields
$E_{s}(\bfsigma)$ and $\Egal(\bfsigma)$ are PAC over $K$.
\end{theorem}

We shall use the last result to find new PAC extensions, and we
start by finding Hilbertian fields over other fields.

\begin{lemma}\label{lem:sHilb_p_t}
Let $K$ be an s-Hilbertian field over a subset $S$ and $E/K$ a
purely transcendental extension. Then $E$ is s-Hilbertian over $S$.
\end{lemma}

\begin{proof}
Let $f_1(\bfT,\bfX),\ldots, f_r(\bfT,\bfX)\in E[\bfT,\bfX]$ be
irreducible polynomials and $0\neq g(\bfT)\in E[\bfT]$. Since $E =
K(u_\alpha \mid \alpha\in A)$, where $\{u_\alpha\mid \alpha\in A\}$
is a set of variables, we can assume that $f_i(\bfT,\bfX) =
g_i(\bfu,\bfT,\bfX)$, where
\[
g_1(\bfu,\bfT,\bfX),\ldots,g_r(\bfu,\bfT,\bfX)\in K[\bfu,\bfT,\bfX]
\]
for some finite tuple of variables $\bfu$.

Since $K$ is s-Hilbertian over $S$, there exists a tuple $\bfa$ of
elements in $S$ such that all $f_i(\bfa,\bfX) = g_i(\bfu,\bfa,\bfX)$
are irreducible in $K[\bfu,\bfX]$ and $g(\bfa)\neq 0$. But the
elements in $\{u_\alpha\mid \alpha\in A\}$ are algebraically
independent, so all $f_i(\bfa,\bfX) = g_i(\bfu,\bfa,\bfX)$ are
irreducible in the larger ring $E[\bfX]$.
\end{proof}

\begin{proposition}\label{prop:HilbertianOver}
Let $K$ be an s-Hilbertian field over a subset $S$ and let $E/K$ be
a finitely generated extension. Then $E$ is Hilbertian over $S$.
Moreover, if $E/K$ is also separable, then $E$ is even s-Hilbertian
over $S$.
\end{proposition}

\begin{proof}
Choose a transcendence basis $\bft$ for $E/K$, i.e., $K(\bft)/K$ is
purely transcendental and $E/K(\bft)$ is finite. Let $H\subseteq
E^r$ be a separable Hilbert set for $E$. By \cite{FriedJarden2005},
Proposition~12.3.3, there exists a separable Hilbert set
$H_1\subseteq K(\bft)^r$ such that $H_1\subseteq H$. By
Lemma~\ref{lem:sHilb_p_t}, we get that $H_1\cap S^r\neq \emptyset$,
and hence the assertion.

If $E/K$ is also separable, then we can choose $\bft$ to be a
separating transcendence basis, that is, we can assume that
$E/K(\bft)$ is separable. Now the same argument as above work for
any Hilbert set $H\subseteq E^r$ (using \cite{FriedJarden2005},
Corollary~12.2.3, instead of Proposition~12.3.3).
\end{proof}

Combining the results that we attained so far, we enlarge the family
of PAC extensions:

\begin{theorem}\label{thm:almostallPACextension}
Let $e\geq 1$ be an integer, let $K$ be a countable field which is
s-Hilbertian over some subset $S$, and let $E/K$ be a finitely
generated extension. Then for almost all $\bfsigma\in \gal(E)^e$ the
fields $E_s(\bfsigma)$ and $\Egal(\bfsigma)$ are PAC over $S$.

In particular, the result is valid when $K$ is a countable
s-Hilbertian field (and $S=K$).
\end{theorem}

\begin{corollary} \label{cor:PACoverf.g.}
Let $e\geq 1$ be an integer and let $K$ be a finitely generated
infinite field (over its prime field). Then for almost all
$\bfsigma\in \gal(K)^e$ the field $K_s(\bfsigma)$ is a PAC extension
of any subfield which is not algebraic over a finite field.
Moreover, if $K$ is of characteristic~$0$, then $K_s(\bfsigma)$ is
also PAC over any subring.
\end{corollary}

\begin{proof}
First assume that $K$ is of characteristic $0$. Then any ring
contains $\bbZ$, so it suffices to show that $K_s(\bfsigma)/\bbZ$ is
PAC for almost all $\bfsigma\in \gal(K)^e$. And indeed, since $\bbQ$
is Hilbertian over $\bbZ$ and $K$ is finitely generated over $\bbQ$,
Theorem~\ref{thm:almostallPACextension} implies that
$K_s(\bfsigma)/\bbZ$ is PAC for almost all $\bfsigma$.

Next assume that the characteristic of $K$ is $p>0$. Since any field
$F$ which is not algebraic over $\bbF_p$ contains a rational
function field $\bbF_p(t)$, it suffices to show that
$K_s(\bfsigma)/\bbF_p(t)$ is PAC for almost all $\bfsigma\in
\gal(K)^e$ and any $t\in K_s(\bfsigma) \smallsetminus \tilde
\bbF_p$.

Set $G=\gal(K)^e$ and let $\mu$ be its normalized Haar measure.
For any $t\in K_s \smallsetminus \tilde\bbF_p$ we define a subset
$\Sigma_{t}\subseteq G$ as follows:
\begin{equation}\label{eq:Sigma_F}
\Sigma_{t} = \{ \bfsigma\in G \mid \mbox{if } t\in K_s(\bfsigma),
\mbox{ then } K_s(\bfsigma)/\bbF_p(t) \mbox{ is PAC}\}.
\end{equation}

We claim that $\mu(\Sigma_{t})=1$. Indeed, let $E=K(t)$. Then $E/K$
is a finite separable extension. Let $H=\gal(E)^e$ be the
corresponding open subgroup of $G$.

Note that $t\in K_s(\bfsigma)$ if and only if $\bfsigma\in H$. Then
the definition of $\Sigma_t$ implies that
\[
\Sigma_t =  (H\cap \Sigma_t)\cup (G\smallsetminus H).
\]
Hence it suffices to show that $\mu(H\cap \Sigma_t)=\mu(H)$, or
equivalently, $\nu(H\cap \Sigma_t)=1$, where $\nu$ denotes the
normalized Haar measure on $H$.

Since $\bbF_p(t)$ is Hilbertian (\cite{FriedJarden2005},
Theorem~13.3.5) and imperfect (\cite{FriedJarden2005}, Lemma~2.7.2),
Uchida's theorem implies that $\bbF_p(t)$ is s-Hilbertian. Also
$E/\bbF_p(t)$ is finitely generated because $K$ is.

Finally, since $H\cap \Sigma_t$ is the set of all $\bfsigma\in
\gal(E)^e$ for which $E_s(\bfsigma)/\bbF_p(t)$ is PAC, and since
$E_s=K_s$, Theorem~\ref{thm:almostallPACextension} implies that
$\nu(H\cap \Sigma_t)=1$, as desired.
\end{proof}

\subsection{The Bottom Theorem}
Now we are ready to address Problem 18.7.8 of
\cite{FriedJarden2005}, the so-called `bottom theorem'. Let $K$ be a
Hilbertian field and $e\geq 1$ an integer. The problem asks whether
for almost all $\bfsigma = (\sigma_1, \ldots, \sigma_e)\in
\gal(K)^e$ the field $M=K_s(\bfsigma)$ has no cofinite proper
subfield (that is, $N\subsetneq M$ implies $[M:N]=\infty$). Here the
phrase `for almost all' refers to the Haar measure on the profinite,
and hence compact, group $\gal(K)$.

Note that the Hilbertian field $K=\bbF_p(t)$ has imperfect degree
$p$, i.e., $[K:K^p]=p$. Moreover, the imperfect degree is preserved
under separable extensions (see \cite{FriedJarden2005},
Lemma~2.7.3), and hence every separable extension $M/K$ satisfies
$[M:M^p]=p$. In particular, $M^p$ is a cofinite proper subfield of
$M=K_s(\bfsigma)$ for all $\bfsigma\in \gal(K)^e$. Consequently, the
problem requires a small modification, namely a separability
assumption:

\begin{conjecture}\label{conj:bottom}
Let $K$ be a Hilbertian field and $e\geq 1$ an integer. Then for
almost all $\bfsigma\in\gal(K)^e$ the field $K_s(\bfsigma)$ is a
finite separable extension of no proper subfield.
\end{conjecture}

In \cite{Haran1985} Haran proves an earlier version of this
conjecture, namely with the additional assumption that $K\subseteq
N$ (see also \cite{FriedJarden2005}, Theorem~18.7.7).

We settle Conjecture~\ref{conj:bottom} in the case $K$ is a finitely
generated infinite field (which is Hilbertian).

\begin{theorem}
Conjecture~\ref{conj:bottom} is true for a finitely generated
infinite field $K$.
\end{theorem}

\begin{proof}
Let $e\geq 1$ be an integer. Corollary~\ref{cor:PACoverf.g.} gives for
almost all $\bfsigma \in \gal(K)^e$ that the field $K_s(\bfsigma)$
is a PAC extension of any subfield of it which is not algebraic over
a finite field. We can also assume that $K_s(\bfsigma)\neq K_s$.

Therefore, if $N\subsetneq K_s(\bfsigma)$ is finite, then
Theorem~\ref{thm:finitePACext} implies that $K_s(\bfsigma)/ N$ is
purely inseparable.
\end{proof}

\subsection{Fields with no Galois Subfields}
We strengthen Chatzidakis's result,
Theorem~\ref{thm:intChatzidakis}, for finitely generated infinite
fields.

\begin{theorem}
Let $K$ be a finitely generated infinite field and $e\geq 1$ an
integer. Then for almost all $\bfsigma \in \gal(K)^e$ the field
$K_s(\bfsigma)$ is a Galois extension of no proper subfield.
\end{theorem}

\begin{proof}
Corollary~\ref{cor:PACoverf.g.} gives for almost all $\bfsigma \in
\gal(K)^e$ that the field $K_s(\bfsigma)$ is a PAC extension of any
subfield of it which is not algebraic over a finite field. We can
also assume that $K_s(\bfsigma)\neq K_s$.

Therefore, Theorem~\ref{thm:intmineGalois} implies that
$K_s(\bfsigma)$ is a Galois extension of no proper subfield.
\end{proof}

\bibliographystyle{amsplain}

\end{document}